\documentclass[11pt]{amsart}
\usepackage{amsmath,amsfonts,mathrsfs}

\newif\ifdraft
\draftfalse
\usepackage{amsmath,amsfonts,amsthm,mathrsfs}
\usepackage{amssymb}

\usepackage[unicode,bookmarks]{hyperref}
\usepackage[usenames,dvipsnames]{xcolor}
\hypersetup{colorlinks=true,citecolor=NavyBlue,linkcolor=Brown,urlcolor=Orange}

\usepackage[alphabetic,initials]{amsrefs}

\usepackage{enumitem}

\usepackage{chngcntr}

\ifdraft
\usepackage[notcite,notref,color]{showkeys}

\definecolor{labelkey}{gray}{0.5}
\fi

\usepackage{tikz}

\usepackage{tikz-cd}

\usetikzlibrary{matrix,arrows}
\newlength{\myarrowsize} 

\pgfarrowsdeclare{cmto}{cmto}{
	\pgfsetdash{}{0pt} 
	\pgfsetbeveljoin 
	\pgfsetroundcap 
	\setlength{\myarrowsize}{0.6pt}
	\addtolength{\myarrowsize}{.5\pgflinewidth}
	\pgfarrowsleftextend{-4\myarrowsize-.5\pgflinewidth} 
	\pgfarrowsrightextend{.8\pgflinewidth}
}{
	\setlength{\myarrowsize}{0.6pt} 
  	\addtolength{\myarrowsize}{.5\pgflinewidth}  
	\pgfsetlinewidth{0.5\pgflinewidth}
	\pgfsetroundjoin
	\pgfpathmoveto{\pgfpoint{1.5\pgflinewidth}{0}}
	\pgfpatharc{-109}{-170}{4\myarrowsize}
	\pgfpatharc{10}{189}{0.58\pgflinewidth and 0.2\pgflinewidth}
	\pgfpatharc{-170}{-115}{4\myarrowsize+\pgflinewidth}
	\pgfpathclose
	\pgfusepathqfillstroke
	\pgfpathmoveto{\pgfpoint{1.5\pgflinewidth}{0}}
	\pgfpatharc{109}{170}{4\myarrowsize}
	\pgfpatharc{-10}{-189}{0.58\pgflinewidth and 0.2\pgflinewidth}
	\pgfpatharc{170}{115}{4\myarrowsize+\pgflinewidth}
	\pgfpathclose
	\pgfusepathqfillstroke
	\pgfsetlinewidth{2\pgflinewidth}
}

\pgfarrowsdeclare{cmonto}{cmonto}{
	\pgfsetdash{}{0pt} 
	\pgfsetbeveljoin 
	\pgfsetroundcap 
	\setlength{\myarrowsize}{0.6pt}
	\addtolength{\myarrowsize}{.5\pgflinewidth}
	\pgfarrowsleftextend{-4\myarrowsize-.5\pgflinewidth} 
	\pgfarrowsrightextend{.8\pgflinewidth}
}{
	\setlength{\myarrowsize}{0.6pt} 
  	\addtolength{\myarrowsize}{.5\pgflinewidth}  
	\pgfsetlinewidth{0.5\pgflinewidth}
	\pgfsetroundjoin
	\pgfpathmoveto{\pgfpoint{1.5\pgflinewidth}{0}}
	\pgfpatharc{-109}{-170}{4\myarrowsize}
	\pgfpatharc{10}{189}{0.58\pgflinewidth and 0.2\pgflinewidth}
	\pgfpatharc{-170}{-115}{4\myarrowsize+\pgflinewidth}
	\pgfpathclose
	\pgfusepathqfillstroke
	\pgfpathmoveto{\pgfpoint{1.5\pgflinewidth}{0}}
	\pgfpatharc{109}{170}{4\myarrowsize}
	\pgfpatharc{-10}{-189}{0.58\pgflinewidth and 0.2\pgflinewidth}
	\pgfpatharc{170}{115}{4\myarrowsize+\pgflinewidth}
	\pgfpathclose
	\pgfusepathqfillstroke
	\pgfpathmoveto{\pgfpoint{1.5\pgflinewidth-0.3em}{0}}
	\pgfpatharc{-109}{-170}{4\myarrowsize}
	\pgfpatharc{10}{189}{0.58\pgflinewidth and 0.2\pgflinewidth}
	\pgfpatharc{-170}{-115}{4\myarrowsize+\pgflinewidth}
	\pgfpathclose
	\pgfusepathqfillstroke
	\pgfpathmoveto{\pgfpoint{1.5\pgflinewidth-0.3em}{0}}
	\pgfpatharc{109}{170}{4\myarrowsize}
	\pgfpatharc{-10}{-189}{0.58\pgflinewidth and 0.2\pgflinewidth}
	\pgfpatharc{170}{115}{4\myarrowsize+\pgflinewidth}
	\pgfpathclose
	\pgfusepathqfillstroke
	\pgfsetlinewidth{2\pgflinewidth}
}

\pgfarrowsdeclare{cmhook}{cmhook}{
	\pgfsetdash{}{0pt} 
	\pgfsetbeveljoin 
	\pgfsetroundcap 
	\setlength{\myarrowsize}{0.6pt}
	\addtolength{\myarrowsize}{.5\pgflinewidth}
	\pgfarrowsleftextend{-4\myarrowsize-.5\pgflinewidth} 
	\pgfarrowsrightextend{.8\pgflinewidth}
}{
	\setlength{\myarrowsize}{0.6pt} 
  	\addtolength{\myarrowsize}{.5\pgflinewidth}  
 	\pgfsetdash{}{0pt}
	\pgfsetroundcap
	\pgfpathmoveto{\pgfqpoint{0pt}{-4.667\pgflinewidth}}
	\pgfpathcurveto
    {\pgfqpoint{4\pgflinewidth}{-4.667\pgflinewidth}}
    {\pgfqpoint{4\pgflinewidth}{0pt}}
    {\pgfpointorigin}
	\pgfusepathqstroke
}

\newenvironment{diagram*}[2]{%
\[%
\begin{tikzpicture}[>=cmto,baseline=(current bounding box.center),%
	to/.style={->,font=\scriptsize,cap=round},%
	into/.style={cmhook->,font=\scriptsize,cap=round},%
	onto/.style={-cmonto,font=\scriptsize,cap=round},%
	math/.style={matrix of math nodes, row sep=#2, column sep=#1,%
		text height=1.5ex, text depth=0.25ex}]%
}{%
\end{tikzpicture}%
\]%
\ignorespacesafterend%
}

\newcommand{\Dmod}{\mathscr{D}}
\newcommand{\Mmod}{\mathcal{M}}

\newcommand{\NN}{\mathbb{N}}
\newcommand{\ZZ}{\mathbb{Z}}

\newcommand{\RR}{\mathbb{R}}
\newcommand{\CC}{\mathbb{C}}
\newcommand{\HH}{\mathbb{H}}

\DeclareMathOperator{\im}{im}

\DeclareMathOperator{\coker}{coker}

\DeclareMathOperator{\codim}{codim}

\DeclareMathOperator{\Pic}{Pic}

\newcommand{\shf}[1]{\mathscr{#1}}

\def\overbar#1#2#3{{%
	\setbox0=\hbox{\displaystyle{#1}}%
	\dimen0=\wd0
	\advance\dimen0 by -#2 
	\vbox {\nointerlineskip \moveright #3 \vbox{\hrule height 0.3pt width \dimen0}%
		\nointerlineskip \vskip 1.5pt \box0}%
}}

\newcommand{\shO}{\shf{O}}

\makeatletter
\let\@@seccntformat\@seccntformat
\renewcommand*{\@seccntformat}[1]{%
  \expandafter\ifx\csname @seccntformat@#1\endcsname\relax
    \expandafter\@@seccntformat
  \else
    \expandafter
      \csname @seccntformat@#1\expandafter\endcsname
  \fi
    {#1}%
}
\newcommand*{\@seccntformat@subsection}[1]{%
  \textbf{\csname the#1\endcsname.}
}
\makeatother

\makeatletter
\let\@paragraph\paragraph
\renewcommand*{\paragraph}[1]{%
	\vspace{0.3\baselineskip}%
	\@paragraph{\textit{#1}}%
}
\makeatother

\counterwithin{equation}{subsection}
\counterwithout{subsection}{section}
\counterwithin{figure}{subsection}

\newtheorem{theorem}[equation]{Theorem}
\newtheorem*{theorem*}{Theorem}
\newtheorem{lemma}[equation]{Lemma}
\newtheorem*{lemma*}{Lemma}
\newtheorem{corollary}[equation]{Corollary}
\newtheorem{proposition}[equation]{Proposition}
\newtheorem*{proposition*}{Proposition}
\newtheorem{conjecture}[equation]{Conjecture}

\theoremstyle{definition}

\newtheorem*{definition*}{Definition}
\newtheorem{remark}[equation]{Remark}

\newtheorem{example}[equation]{Example}
\newtheorem*{example*}{Example}
\newtheorem*{problem*}{Problem}

\theoremstyle{plain}

\newcommand{\theoremref}[1]{\hyperref[#1]{Theorem~\ref*{#1}}}
\newcommand{\lemmaref}[1]{\hyperref[#1]{Lemma~\ref*{#1}}}
\newcommand{\definitionref}[1]{\hyperref[#1]{Definition~\ref*{#1}}}
\newcommand{\propositionref}[1]{\hyperref[#1]{Proposition~\ref*{#1}}}
\newcommand{\conjectureref}[1]{\hyperref[#1]{Conjecture~\ref*{#1}}}
\newcommand{\corollaryref}[1]{\hyperref[#1]{Corollary~\ref*{#1}}}
\newcommand{\exampleref}[1]{\hyperref[#1]{Example~\ref*{#1}}}

\makeatletter
\let\old@caption\caption
\renewcommand*{\caption}[1]{%
	\setcounter{figure}{\value{equation}}%
	\stepcounter{equation}%
	\old@caption{#1}\relax%
}
\makeatother

\newcounter{intro}

\newtheorem{intro-conjecture}[intro]{Conjecture}
\newtheorem{intro-corollary}[intro]{Corollary}
\newtheorem{intro-theorem}[intro]{Theorem}

\newcommand{\parref}[1]{\hyperref[#1]{\S\ref*{#1}}}

\makeatletter
\newcommand*\if@single[3]{%
  \setbox0\hbox{${\mathaccent"0362{#1}}^H$}%
  \setbox2\hbox{${\mathaccent"0362{\kern0pt#1}}^H$}%
  \ifdim\ht0=\ht2 #3\else #2\fi
  }
\newcommand*\rel@kern[1]{\kern#1\dimexpr\macc@kerna}
\newcommand*\widebar[1]{\@ifnextchar^{{\wide@bar{#1}{0}}}{\wide@bar{#1}{1}}}
\newcommand*\wide@bar[2]{\if@single{#1}{\wide@bar@{#1}{#2}{1}}{\wide@bar@{#1}{#2}{2}}}
\newcommand*\wide@bar@[3]{%
  \begingroup
  \def\mathaccent##1##2{%
    \if#32 \let\macc@nucleus\first@char \fi
    \setbox\z@\hbox{$\macc@style{\macc@nucleus}_{}$}%
    \setbox\tw@\hbox{$\macc@style{\macc@nucleus}{}_{}$}%
    \dimen@\wd\tw@
    \advance\dimen@-\wd\z@
    \divide\dimen@ 3
    \@tempdima\wd\tw@
    \advance\@tempdima-\scriptspace
    \divide\@tempdima 10
    \advance\dimen@-\@tempdima
    \ifdim\dimen@>\z@ \dimen@0pt\fi
    \rel@kern{0.6}\kern-\dimen@
    \if#31
      \overline{\rel@kern{-0.6}\kern\dimen@\macc@nucleus\rel@kern{0.4}\kern\dimen@}%
      \advance\dimen@0.4\dimexpr\macc@kerna
      \let\final@kern#2%
      \ifdim\dimen@<\z@ \let\final@kern1\fi
      \if\final@kern1 \kern-\dimen@\fi
    \else
      \overline{\rel@kern{-0.6}\kern\dimen@#1}%
    \fi
  }%
  \macc@depth\@ne
  \let\math@bgroup\@empty \let\math@egroup\macc@set@skewchar
  \mathsurround\z@ \frozen@everymath{\mathgroup\macc@group\relax}%
  \macc@set@skewchar\relax
  \let\mathaccentV\macc@nested@a
  \if#31
    \macc@nested@a\relax111{#1}%
  \else
    \def\gobble@till@marker##1\endmarker{}%
    \futurelet\first@char\gobble@till@marker#1\endmarker
    \ifcat\noexpand\first@char A\else
      \def\first@char{}%
    \fi
    \macc@nested@a\relax111{\first@char}%
  \fi
  \endgroup
}
\makeatother

\newcommand{\I}{\mathcal{I}}

\DeclareMathOperator{\tot}{Tot}

\def\C{{\mathcal C}}

\def\ZZ{{\mathbf Z}}
\def\NN{{\mathbf N}}
\def\CC{{\mathbf C}}
\def\AAA{{\mathbf A}}
\def\RR{{\mathbf R}}

\def\Pic{{\rm Pic}}

\begin{document}

\vspace{\baselineskip}

\title{Local vanishing and Hodge filtration for rational singularities}

\author[M. Musta\c{t}\v{a}]{Mircea~Musta\c{t}\u{a}}
\address{Department of Mathematics, University of Michigan,
Ann Arbor, MI 48109, USA}
\email{{\tt mmustata@umich.edu}}

\author[S.~Olano]{Sebasti\'an~Olano}
\address{Department of Mathematics, Northwestern University, 
2033 Sheridan Road, Evanston, IL
60208, USA} \email{{\tt seolano@math.northwestern.edu}}

\author[M.~Popa]{Mihnea~Popa}
\address{Department of Mathematics, Northwestern University, 
2033 Sheridan Road, Evanston, IL
60208, USA} \email{{\tt mpopa@math.northwestern.edu}}

\thanks{MM was partially supported by NSF grant DMS-1401227; MP was partially supported by NSF grant
DMS-1405516.}


\begin{abstract}
Given an $n$-dimensional variety $Z$ with rational singularities, we conjecture that if
$f\colon Y\to Z$ is a resolution of singularities whose reduced exceptional divisor $E$ has
simple normal crossings, then 
$$R^{n-1}f_*\Omega_Y(\log E)=0.$$
We prove this when $Z$ has isolated singularities and when it is a toric variety.
We deduce that for a divisor $D$
with isolated rational singularities on a smooth complex $n$-dimensional variety $X$, 
the generation level of Saito's Hodge filtration on the localization $\shO_X(*D)$ is at most $n-3$.
\end{abstract}

\maketitle

\section{Introduction}

We propose the following local vanishing conjecture for  log resolutions of varieties with rational singularities:

\begin{intro-conjecture}\label{conj_vanishing}
If $Z$ is a complex variety of dimension $n\geq 2$, with rational singularities, and $f\colon Y\to Z$ is a resolution of singularities whose reduced exceptional divisor $E$ has simple normal crossings, then
$$
R^{n-1}f_*\Omega_Y(\log E)=0.
$$
\end{intro-conjecture}

The related local vanishing 
$$
R^{n-1}f_*\big( \Omega_Y(\log E) \otimes \shO_Y (- E)\big)=0
$$
is already known; it is a variant of the Steenbrink-type vanishing theorem \cite[Theorem~14.1]{GKKP}, as explained in 
\S\ref{GKKP}.

The main purpose of this paper is to answer in the affirmative the case of isolated singularities.

\begin{intro-theorem}\label{main}
Conjecture~\ref{conj_vanishing} holds when $Z$ has isolated singularities.
\end{intro-theorem}

The proof relies on results from both birational geometry and Hodge theory. One ingredient is the 
Steenbrink-type vanishing from \cite{GKKP} mentioned above. With the help of this theorem, we reduce our statement to a problem in Hodge theory. In the case of surfaces, it can be solved using the Hodge Index theorem. In higher dimension however, the solution relies on more subtle results of de Cataldo-Migliorini \cite{cm05}, \cite{cm07} on the Hodge theory of algebraic maps, combined with rudiments of mixed Hodge theory. 

We also show the following statement, relying on standard facts from the theory of toric varieties.

\begin{intro-theorem}\label{thm_toric}
Conjecture~\ref{conj_vanishing} holds when $Z$ is a toric variety.
\end{intro-theorem}

One source of interest in Conjecture \ref{conj_vanishing} is the fact that, according to a criterion in \cite{MP}, it leads to a bound on the generation level of Saito's Hodge filtration for hypersurfaces with rational singularities. Given a smooth complex variety 
$X$, and a reduced divisor $D$ on $X$, let $\shO_X (*D)$ be the $\Dmod_X$-module of rational functions with poles along $D$, i.e. the localization of $\shO_X$ along $D$.
Saito's theory of mixed Hodge modules \cite{Saito-MHM} endows it with a Hodge filtration $F_k \shO_X (*D)$, $k\ge 0$, compatible with the standard filtration on $\Dmod_X$,
where $F_{\ell} \Dmod_X$ consists of differential operators of order at most $\ell$.

Saito introduced in \cite{Saito-HF} a measure of the complexity of this filtration; one says that it is \emph{generated at level $k$} if 
$$F_{\ell} \Dmod_X \cdot F_k\shO_X(*D)=F_{k+\ell}\shO_X(*D)\quad\text{for all}\quad \ell\geq 0.$$
The smallest integer $k$ with this property is called the \emph{generating level}.
It was shown in \cite[Theorem~B]{MP} that if $X$ has dimension $n \ge 2$, then $F_\bullet \shO_X(*D)$ is always generated at level $n-2$. This bound is sharp even when 
$n \ge 3$; see e.g. \cite[Example~17.9]{MP}. We propose an improvement in the case of rational singularities:

\begin{intro-conjecture}\label{conj_filtration}
If $D$ has only rational singularities and $n \ge 3$, 
then the Hodge filtration $F_\bullet \shO_X (*D)$ is generated at level $n -3$.
\end{intro-conjecture}

When $D$ has an isolated quasihomogeneous singularity, a stronger bound was given by Saito in \cite[Theorem~0.7]{Saito-HF}: the generating level of $F_\bullet \shO_X(*D)$ is $[n - \alpha_f] - 1$, 
where $\alpha_f$ is the microlocal log
canonical threshold of $D$, i.e. the negative of the largest root of its reduced Bernstein-Sato polynomial. It is known that the singularity being rational is equivalent to $\alpha_f > 1$; see \cite[Theorem~0.4]{Saito-B}.  We note that for isolated semiquasihomogeneous singularities the generating level can be even lower. In particular, the example in Remark (i) after 
5.4 in \cite{Saito-HF} provides a singularity which is not rational (as $\alpha_f < 1$), but with generating level at most $n-3$. This shows that the converse of the statement of Conjecture \ref{conj_filtration} 
is not true in general.

A consequence of Theorem \ref{main} is the fact that Conjecture \ref{conj_filtration} holds whenever the divisor $D$ has isolated singularities. More precisely, we show the following:

\begin{intro-theorem}\label{thm_filtration}
Conjecture~\ref{conj_filtration} is equivalent to Conjecture~\ref{conj_vanishing} when $Z$ is a hypersurface.
In particular, Conjecture~\ref{conj_filtration} holds when the divisor $D$ has isolated singularities.
\end{intro-theorem}

It is natural to ask more boldly whether Saito's formula $[n - \alpha_f] - 1$ for the generating level holds for all rational singularities. 

We also propose in Theorem \ref{thm_m_adic} a reduction of the full statement of Conjecture \ref{conj_filtration} to the case of isolated singularities treated here. It is based on a conjectural statement of independent interest regarding Hodge ideals \cite{MP}, an alternative way of approaching the study of $F_\bullet \shO_X(*D)$.
More precisely, the statement is about their $\frak{m}$-adic approximation, and is known to hold for multiplier ideals; 
see Conjecture \ref{conj_m_adic} and Example \ref{multiplier}. 

\medskip

\noindent
{\bf Acknowledgements.}
We thank Mark Andrea de Cataldo for useful conversations and S\'{a}ndor Kov\'acs for his comments on a preliminary version of this paper. We are also grateful to the referee for suggesting a quicker and more conceptual argument for the main result in 
\S\ref{scn:beta}.

\section{The proof for isolated singularities}

Our goal in this section is to prove Conjecture~\ref{conj_vanishing} in the case of varieties with isolated singularities. 

\subsection{Preliminaries}
We fix a variety $Z$ with rational singularities and a resolution $f\colon Y\to Z$
as in Conjecture~\ref{conj_vanishing}, with reduced exceptional divisor $E=\sum_{i=1}^d E_i$, where the $E_i$ are mutually distinct prime divisors.

\begin{lemma}\label{lem1_conj_vanishing}
The assertion in Conjecture~\ref{conj_vanishing} is independent of the chosen resolution.  
\end{lemma}

\begin{proof}
Given any two resolutions as in the statement, we can find one that dominates both. 
Therefore it is enough to consider the case when $g\colon W\to Y$ is such that $h=f\circ g$ is another 
resolution of $Z$ whose reduced exceptional divisor $F$ has simple normal crossings. Note that in this case
$F$ is the sum of the strict transform of $E$ and the $g$-exceptional divisor. Therefore, since
$Y$ is smooth and $E$ has simple normal crossings, we deduce from 
\cite[Theorem~31.1(ii)]{MP} that
$$f_*\Omega_W(\log F)\simeq \Omega_Y(\log E)$$
and
$$R^if_*\Omega_W(\log F)=0\quad\text{for all}\quad i>0.$$
 The Leray spectral sequence then gives
$$R^qh_*\Omega_W(\log F)\simeq R^qf_*\Omega_Y(\log E)\quad\text{for all}\quad q\geq 0,$$
which implies the assertion.
\end{proof}

\begin{remark}
Note that Lemma~\ref{lem1_conj_vanishing} implies in particular that Conjecture~\ref{conj_vanishing} holds when $Z$ is smooth. Indeed, it allows us to take $f$ to be the identity, in which case the assertion is clear.
\end{remark}

We now begin the preparations for the proof of Theorem~\ref{main}. 
By Lemma~\ref{lem1_conj_vanishing}, the vanishing in Conjecture~\ref{conj_vanishing} does not depend on $f$.
Hence we may and will assume that $f$ is a composition of blow-ups with centers lying over the singular locus $Z_{\rm sing}$ of $Z$.
We also assume that
the exceptional locus of $f$ has pure codimension $1$ (and it is thus equal to the support of $E$).
In particular, $f$ is an isomorphism over $Z\smallsetminus Z_{\rm sing}$, and thus $E$ lies over $Z_{\rm sing}$.
The assertion is also local on $Z$, hence without loss of generality we may and will assume that $Z$ is affine.
We will identify coherent sheaves on $Z$ with their spaces of global sections.

\subsection{A reformulation of the problem}
We have the following:

\begin{lemma}\label{vanishing_O}
With the above notation, we have
$$H^{n-1} (E_i,\shO_{E_i})=0 \quad\text{for all}\quad 1 \le i \le d.$$
\end{lemma}

\begin{proof}
Since $Z$ has rational singularities, we have 
$$H^{n-1}(Y,\shO_Y)=0,$$ 
while the fact that $f$ has fibers of dimension $\leq n-1$ implies 
$$H^{n}\big(Y, \shO_{Y}(-E_i)\big)=0\quad\text{for all}\quad i.$$
Passing to cohomology in the short exact sequence
$$0\longrightarrow  \shO_Y(-E_i)\longrightarrow \shO_Y\longrightarrow \shO_{E_i}\longrightarrow 0$$
implies then the statement.
\end{proof}

Consider now on $Y$ the residue short exact sequence 
$$0\longrightarrow\Omega_Y\longrightarrow\Omega_Y (\log E)\longrightarrow\bigoplus_{i=1}^d\shO_{E_i}
\longrightarrow 0.$$
It follows from the corresponding long exact sequence and Lemma~\ref{vanishing_O}  that 
we can rephrase the vanishing predicted by Conjecture~\ref{conj_vanishing} (when $Z$ has rational singularities),
as follows:

\begin{proposition}\label{reformulation2}
With the above notation, we have
$$H^{n-1}\big(Y, \Omega_Y(\log E)\big)=0$$
if and only if the connecting homomorphism
$$
\alpha\colon \bigoplus_{i=1}^d H^{n-2}(E_i,\shO_{E_i})\longrightarrow H^{n-1}(Y, \Omega_Y)
$$
is surjective.
\end{proposition}

\subsection{A vanishing theorem for log canonical pairs}\label{GKKP}
We continue to assume that $Z$ has rational singularities. 
In this case, the Steenbrink-type vanishing theorem
\cite[Theorem~14.1]{GKKP} gives
\begin{equation}\label{eqn3_theorem}
H^{n-1}\big(\Omega_Y(\log E)\otimes\shO_Y(-E)\big)=0.
\end{equation}
We note that the result in \emph{loc. cit.} is stated for log canonical pairs $(Z,D)$. However, when $D=0$,
the result also holds if we only assume that $Z$ has Du Bois singularities (this is the only condition that is used in the
proof, via \cite[Theorem~13.3]{GKKP}). In our case this condition is satisfied since rational singularities are Du Bois by \cite[Theorem~S]{Kovacs}.

Note that we also have
\begin{equation}\label{eqn7_theorem}
H^n\big(\Omega_Y(\log E)\otimes\shO_Y(-E)\big)=0,
\end{equation}
due to the fact that all fibers of $f$ have dimension $\leq n-1$. These two vanishing statements will be used 
later in combination with Proposition \ref{reformulation2}.

\subsection{A complex describing $\Omega_Y(\log E) (-E)$}
In order to make use of Proposition \ref{reformulation2} we will need the following, likely familiar to experts:

\begin{lemma}\label{general_exact_sequence}
Suppose that $E=\sum_{i=1}^d E_i$ is a simple normal crossing divisor on the smooth, $n$-dimensional variety $Y$.
If for every $J\subseteq\{1,\ldots,d\}$ we denote
$$E_J:=\bigcap_{i\in J}E_i,$$
then there is an exact complex
$$0\to \Omega_{Y}(\log E)\otimes\shO_Y (-E)\to {\mathcal C}^0=\Omega_{Y}\overset{d^1}\to {\mathcal C}^1\to\ldots\overset{d^{n-2}}\to {\mathcal C}^{n-1}\to 0,$$
where
$${\mathcal C}^p=\bigoplus_{|J|=p}\Omega_{E_J} \quad \text{for all} \quad 1\leq p\leq n-1,$$
and the maps $d^i$ are induced, up to sign, by the obvious restriction maps.
\end{lemma}

\begin{proof}
This is a local assertion, hence we may assume that we have an algebraic  system of coordinates $x_1,\ldots,x_n$ on $Y$ such that
$E_i$ is defined by $x_i$ for $1\leq i\leq d$. The coordinates $x_1,\ldots,x_d$ define a smooth map $\varphi\colon Y\to\AAA^d$ such that $E=\varphi^*H$,
where $H$ is the sum of the coordinate hyperplanes. Since exactness is preserved by flat pull-back, it is enough to prove the lemma when $Y=\AAA^d$ and
$E_i$ is defined by $x_i$. 

In this case, all the terms in the complex carry a natural $\NN^d$-grading (where each $dx_i$ has degree $0$), with the maps preserving the grading. Therefore it is enough to check exactness in each degree.
Note that the kernel of 
$$\Omega_{\AAA^d}\longrightarrow \bigoplus_{i=1}^d\Omega_{E_i}$$
consists of those $\sum_if_idx_i$ such that $f_i$ is divisible by $x_j$ for every $j\neq i$. Therefore this kernel is precisely
$\Omega_{Y}(\log E)\otimes\shO_Y (-E)$. Consequently we only need to check the exactness of the complex in the lemma at each 
${\mathcal C}^i$, with $1\leq i\leq d-1$.

Let's consider $u=(u_1,\ldots,u_d)\in\NN^d$. Note that
$${\mathcal C}^p_u=\bigoplus_J\bigoplus_j\CC x^udx_j,$$
where the sum is taken over those subsets $J\subseteq \{1,\ldots,d\}$ with $|J|=p$
and such that $u_i=0$ for all $i\in J$, and over all $j\not\in J$. 
Equivalently, $j$ runs over $\{1,\ldots,d\}$ and for every $j$, the set $J$ varies over the
subsets of $\{i\in\{1,\ldots,d\}\mid u_i=0\}\smallsetminus \{j\}$ with $p$ elements.
We thus see that the degree $u$ component of the complex
$$0\longrightarrow {\mathcal C}^0\longrightarrow {\mathcal C}^1\longrightarrow\ldots\longrightarrow {\mathcal C}^{d-1}\longrightarrow 0$$
is a direct sum of $d$ complexes, each of them isomorphic to the complex computing the reduced simplicial cohomology
of the full simplicial complex on a suitable set. Each such complex has no cohomology in positive degrees (and it has cohomology 
in degree $0$ if and only if the corresponding set is empty). This proves the exactness of the complex in the lemma at each 
${\mathcal C}^i$, for $i\geq 1$.
\end{proof}

\subsection{The proof of Conjecture~\ref{conj_vanishing} for $n=2$}
When the dimension of $Z$ is $2$, the required vanishing 
is easy to obtain using the reformulation in Proposition~\ref{reformulation2}. In this case the complex in Lemma~\ref{general_exact_sequence} is simply
 $$0\longrightarrow \Omega_Y(\log E)\otimes\shO_Y (-E)\longrightarrow \Omega_Y\longrightarrow\bigoplus_{i=1}^d\Omega_{E_i}\longrightarrow 0.$$
 Using  (\ref{eqn3_theorem}) and (\ref{eqn7_theorem}), we see that the induced map
 \begin{equation}\label{eq_isom}
 \beta\colon H^1(Y, \Omega_Y)\longrightarrow \bigoplus_{i=1}^d H^1(E_i,\Omega_{E_i})
 \end{equation}
 is an isomorphism. 
 
 On the other hand, note that $\alpha$ in Proposition~\ref{reformulation2} 
 maps the element $1\in H^0(E_i, \shO_{E_i})\simeq\CC$ to
the class ${\rm cl}(E_i)$, that is to the image of $\shO(E_i)$ via the map
$$\Pic(Y)\longrightarrow H^1(Y,\Omega_Y)$$
induced by $\shO_Y^*\to\Omega_Y$, $u\to {\rm dlog}(u)$. 
Furthermore, it is well-known (see, for example, \cite[Exercise~V.1.8]{Hartshorne}) that the image of
${\rm cl}(E_i)$ in $H^1(E_j,\Omega_{E_j})\simeq \CC$ is the intersection product $(E_i\cdot E_j)$. We conclude that, via the isomorphism (\ref{eq_isom}), 
the map $\alpha$
is given by the matrix $(E_i\cdot E_j)_{1\leq i,j\leq d}$. The fact that this matrix is non-singular (in fact, negative definite) is a well-known consequence of
the Hodge Index theorem.

\subsection{The set-up in higher dimension}
From now on we assume that $n\geq 3$. We also assume that $Z$ has isolated singularities and in fact,
after restricting to suitable affine open subsets, that $Z_{\rm sing}$ is a point 
and that $E$ lies over it. 
In particular all $E_i$ are smooth projective varieties, of dimension $n-1$. We consider the morphism 
$$\beta\colon H^{n-1}(Y, \Omega_Y)\longrightarrow \bigoplus_{i=1}^d H^{n-1}(E_i,\Omega_{E_i})$$
induced by the map ${\mathcal C}^0\to {\mathcal C}^1$ in Lemma~\ref{general_exact_sequence}. 
For $p\geq 1$, we also consider 
$$\Mmod^p:={\rm ker}(d^{p+1})\subseteq {\mathcal C}^p.$$
The vanishing statements (\ref{eqn3_theorem}) and (\ref{eqn7_theorem}) imply that the map ${\mathcal C}^0\to \Mmod^1$
induces an isomorphism 
\begin{equation}\label{eq1_comments}
H^{n-1}(Y,\Omega_Y)\simeq H^{n-1}(Y,\Mmod^1).
\end{equation}
Note that for every $p$ we have
$$\dim {\rm Supp}(\Mmod^p)\leq \dim {\rm Supp}({\mathcal C}^p)=n-p.$$
In particular, from the exact sequence
$$0\to\Mmod^1\to {\mathcal C}^1\to\Mmod^2\to 0$$
 we deduce that the induced morphism 
$$\varphi^1\colon H^{n-1}(Y,\Mmod^1)\to H^{n-1}(Y,{\mathcal C}^1)$$
is surjective. By combining this with (\ref{eq1_comments}), we conclude that $\beta$ is always surjective. 

On the other hand, it follows from Poincar\'{e} duality and Hodge symmetry that for every $i$ with $1\leq i\leq d$, we have
$$h^{n-2}(E_i,\shO_{E_i})=h^{0,n-2}(E_i)=h^{n-1,1}(E_i)=h^{1,n-1}(E_i)=h^{n-1}(E_i,\Omega_{E_i}).$$
Therefore the source and target of $\beta\circ\alpha$ have the same dimension. We deduce the following:

\begin{lemma}\label{equivalence}
With the above notation, the following are equivalent:
\begin{enumerate}
\item[i)] $\alpha$ is surjective.
\item[ii)] $\alpha$ and $\beta$ are isomorphisms.
\item[iii)] $\alpha$ and $\beta$ are injective.
\end{enumerate}
\end{lemma}

\begin{proof}
Note that if $\alpha$ is surjective, since $\beta$ is also surjective, we conclude that $\beta\circ\alpha$ is surjective as well,
hence it is an isomorphism. This implies that $\alpha$ is injective, hence an isomorphism, and therefore $\beta$ is an isomorphism as well.
The other implications are clear.
\end{proof}

\subsection{The map $\beta$.}\label{scn:beta}
In order to simplify the notation, we define 
$$\displaystyle E(p) := \coprod_{|J| = p }E_J,$$ 
with the convention that $E(0) = Y$. Thus in Lemma \ref{general_exact_sequence} we have $\C^p = \Omega_{E(p)}$
for $0\leq p\leq n-1$. We reinterpret the map $\beta$ as
$$\beta\colon H^{n-1}(Y, \Omega_Y) \longrightarrow H^{n-1}(E(1), \Omega_{E(1)}).$$

\begin{proposition}\label{beta}
With the above notation, if $n\geq 3$, then $\beta$
is an isomorphism.
\end{proposition}

Before giving the proof of the proposition, we make some preparations. All cohomology groups below are considered to 
be with complex coefficients.
Recall that  for a simple normal crossing divisor $E$ as above, the weight $k$ piece of the mixed Hodge structure 
on the cohomology of $E$ can be computed using the complex
$$ 0 \longrightarrow H^k\big(E(1)\big) \stackrel{\delta_1}{\longrightarrow} H^k\big(E(2)\big) \stackrel{\delta_2}{\longrightarrow} \cdots \stackrel{\delta_l}{\longrightarrow} 
H^k\big(E(l+1)\big) \stackrel{\delta_{l+1}}{\longrightarrow} \cdots.$$
(See e.g. \cite[Part II, 1]{elz83}.)
More precisely, we have  $${\rm Gr}^W_k H^{k+l}(E) = \ker\delta_{l+1}/\im{\delta_l}.$$ 
The Hodge space $H^{p,q}\big({\rm Gr}^W_k H^{k+l}(E)\big)$  is obtained by applying $H^{p,q}$ to this complex and passing 
to cohomology, as above.  

The following result of Steenbrink \cite[Corollary 1.12]{St83} is crucial in what follows:

\begin{lemma}\label{lemmaps}
Let $Z$ be an algebraic variety of dimension $n$ with an isolated singularity $x\in Z$. If $f \colon Y \to Z$ is a resolution 
such that $f^{-1}(x) = E$ is a simple normal crossing divisor and $f$ is an isomorphism over $Z\smallsetminus\{x\}$, then
 $${\rm Gr}_r^W H^k(E)= 0 \quad \text{ for } r\neq k \quad \text{  if } \quad k\geq n.$$
In other words,  $H^k(E)$ has a pure Hodge structure of weight $k$ for $k \ge n$.
\end{lemma}

Next, in the notation of Lemma \ref{general_exact_sequence}, we set
$${\mathcal C}^{\geq 1} := {\mathcal C}^1 \to \cdots \to {\mathcal C}^{n-1} \to 0.$$ 
By definition of ${\mathcal M}^1$ in the previous section, the map
 ${\mathcal M}^1 \to {\mathcal C}^1$ induces a quasi-isomorphism 
 $$ {\mathcal M}^1 \overset{\sim}{\longrightarrow} {\mathcal C}^{\geq 1},$$
and so 
\begin{equation}\label{equation_cohomology} 
H^{n-1}(Y, {\mathcal M}^1) \simeq \HH^{n-1}(Y, {\mathcal C}^{\geq 1}). 
\end{equation}

On the other hand, we have: 

\begin{lemma}\label{lemma_hodge_filtration} 
With the notation above $$\HH^{n-1}(Y, {\mathcal C}^{\geq 1}) \simeq {\rm Gr}_F^1H^n(E).$$
\end{lemma}

\begin{proof}
Recall that the Hodge filtration on $\tot\big(\Omega^\bullet_{E(\cdot)}\big)$ is defined as 
$$F^k = \tot(\tau^{\geq k}\Omega^\bullet_{E(\cdot)}),$$ 
where $\tot$ is the total complex associated to a double complex, and $\tau^{\ge k}$ denotes the standard truncation at the $k$-th term (see e.g. \cite[Section 3.2.4.2]{Elzeinetal}). This means that 
$${\rm Gr}_F^1\tot\big(\Omega^\bullet_{E(\cdot)}\big) = \big(0 \to \Omega^1_{E(1)} \to \cdots \to \Omega^1_{E(n-1)} \to 0\big) = {\mathcal C}^{\geq 1}[-1].$$ 
The Hodge filtration of $H^n(E, \CC)$ is defined by this filtration, together with the quasi-isomorphism 
$$ \CC_E \simeq \tot\big(\Omega^\bullet_{E(\cdot)}\big).$$ 
Moreover, the spectral sequence associated to this filtration degenerates at $E_1$ (see e.g. \cite[Theorem 3.3]{elz83}), which implies: 
$$ \HH^n\big(Y, {\rm Gr}_F^1\tot\big(\Omega^*_{E(\cdot)}\big)\big) \simeq {\rm Gr}_F^1\HH^n\big(Y, \tot\big(\Omega^*_{E(\cdot)}\big)\big).$$ 
Using the descriptions above, we deduce the isomorphism 
$$\HH^{n-1}(Y, {\mathcal C}^{\geq 1}) \simeq {\rm Gr}_F^1H^n(E).$$
\end{proof}

\begin{proof}[Proof of Proposition \ref{beta}]\footnote{We thank the referee for suggesting this approach, which is shorter and 
more conceptual than our original proof.}
We have already seen that the map $$\varphi^1: H^{n-1}(Y, {\mathcal M}^1) \to H^{n-1}(Y, {\mathcal C}^1)$$ is surjective. 
On the other hand, by (\ref{equation_cohomology}) and Lemma~\ref{lemma_hodge_filtration}, 
$$H^{n-1}(Y, {\mathcal M}^1) \simeq {\rm Gr}_F^1H^n(E).$$ 
As $H^n(E)$ has a pure Hodge structure of weight $n$ by Lemma~\ref{lemmaps}, and $E$ is $n-1$ dimensional, 
$${\rm Gr}_F^1H^n(E) \simeq H^{1,n-1}(H^n(E)).$$ 
For dimension reasons, the $H^{1,n-1}$ piece of the complex computing the weight $n$ piece of the Hodge structure on the cohomology of $E$ is 
$$0\to H^{1,n-1}(E(1)) \to 0,$$ 
hence
$$H^{1,n-1}(H^n(E)) = H^{1,n-1}({\rm Gr}^W_nH^n(E)) \simeq H^{1, n-1}(E(1)).$$  
Finally, recall that $$H^{n-1}(Y, {\mathcal C}^1) = H^{n-1}(E(1), \Omega^1_{E(1)}) \simeq H^{1,n-1}(E(1)).$$ 
Since $\varphi^1$ is a surjective morphism between two vector spaces that are abstractly isomorphic, it follows that
$\varphi^1$ is an isomorphism. As we have seen in ($\ref{eq1_comments}$),  the map $$H^{n-1}(Y, \Omega_Y)  \longrightarrow H^{n-1}(Y, \Mmod^1)$$ 
is also an isomorphism. The composition of these two maps is $\beta$, which is thus an isomorphism too.
\end{proof}

\subsection{The map $\alpha$.}
Since we have seen in Proposition \ref{beta} that $\beta$ is an isomorphism, Lemma \ref{equivalence} implies that in order to finish the proof of Theorem~\ref{main}, it suffices to show that $\alpha$ is injective. This is equivalent to the following:

\begin{proposition}\label{alpha}
The map $\beta\circ\alpha\colon H^{0,n-2}\big(E(1)\big) \to H^{1,n-1}\big(E(1)\big)$ is an isomorphism.
\end{proposition}

Note that since $Z$ has an isolated singularity, by possibly restricting to an open affine as before, we may assume that $Z$ is an open subset of a projective variety $\overline{Z}$ such that
$\overline{Z}\smallsetminus Z_{\rm sing}$ is smooth. Indeed, if $Z$ is affine, we may choose an open embedding $Z\hookrightarrow W$, with $W$
a projective variety. Consider a resolution of singularities $\varphi\colon V\to W\smallsetminus Z_{\rm sing}$ given by a composition of blow-ups with centers over the singular
 locus of $W\smallsetminus Z_{\rm sing}$. In particular, $\varphi$ is an isomorphism over $Z\smallsetminus Z_{\rm sing}$. 
By blowing up $W$ along the same sequence of centers, we obtain 
a projective variety $\overline{Z}$
in which $Z$ embeds as an open subset and such that $\overline{Z}\smallsetminus Z_{\rm sing}$ is smooth. 

Recall that the morphism $f\colon Y\to Z$ is a composition of blow-ups with centers lying over $Z_{\rm sing}$. By blowing up $\overline{Z}$ 
along the same sequence of centers, we obtain a smooth, projective variety $\overline{Y}$, 
with a morphism $g\colon \overline{Y}\to \overline{Z}$ which is an isomorphism over $\overline{Z}\smallsetminus Z_{\rm sing}$.
Note that $f$ is obtained by restricting $g$ to $Y=g^{-1}(Z)$.

We have a commutative diagram
$$
\begin{tikzcd}
H^{n-2}\big(E(1),\shO_{E(1)}\big)\rar{\overline{\alpha}}\dar{\rm Id} & H^{n-1}(\overline{Y},\Omega_{\overline{Y}})\dar\rar{\overline{\beta}} &
H^{n-1}\big(E(1),\Omega_{E(1)}\big)\dar{\rm Id}\\
H^{n-2}\big(E(1),\shO_{E(1)}\big)\rar{\alpha} & H^{n-1}(Y,\Omega_Y)\rar{\beta} & H^{n-1}\big(E(1),\Omega_{E(1)}\big),
\end{tikzcd}
$$
in which the middle vertical map is the pull-back induced by inclusion, and $\overline{\alpha}$, $\overline{\beta}$ are defined in the same way as  $\alpha$, $\beta$
(but considering $E$ as divisor on the variety $\overline{Y}$).

Note that the map
$$\overline{\alpha}\colon H^{n-2}\big(E(1), \shO_{E(1)}\big) \longrightarrow H^{n-1}(\overline{Y}, \Omega_{\overline{Y}})$$
 is a Gysin map. It can be seen as a direct summand in the composition
$$ H^{n-2}\big(E(1)\big) \stackrel{P.D.}{\longrightarrow} H_{n}\big(E(1)\big) \stackrel{i_*}{\longrightarrow} H_{n}(\overline{Y}) \stackrel{P.D.}{\longrightarrow} H^{n}(\overline{Y}),$$ 
where $i \colon E(1) \hookrightarrow \overline{Y}$ is the inclusion map on each of the components, and the external maps are isomorphisms given by Poincar\'e duality.

\begin{example}
We treat the case $n=3$ first. In \cite{cm07}, the authors define an intersection pairing on $H^3(E)$. Indeed, in \S2.2 in \emph{loc. cit.}
the case $l=0$, which means $E$ is a fiber as in our situation, corresponds to a pairing given by 
$$H_3(E) \stackrel{j_*}{\longrightarrow} H_3(\overline{Y}) \stackrel{P.D.}{\longrightarrow} H^3(\overline{Y}) \stackrel{j^*}{\longrightarrow} H^3(E) \simeq (H_3(E))^*$$ where 
$j\colon E \hookrightarrow \overline{Y}$ is the inclusion. Let $T \colon H_3(E) \to H^3(E)$ be this composition.
By \cite[Corollary 2.3.6]{cm07}  this pairing is nondegenerate (that is, $T$ is an isomorphism), and our task is to relate it to the cohomology of $E(1)$.

As stated earlier, and also proved in \cite{cm07}, $H^3(E)$ has a pure weight $3$ Hodge structure. Given that $E(2)$ is 1-dimensional, we obtain that the complex calculating the third graded piece of the mixed Hodge structure of $E$ is simply  $0\to H^3\big(E(1)\big) \to 0$, and therefore we get an
isomorphism $H^3(E) \simeq H^3\big(E(1)\big)$, induced by the canonical map $E(1)\to E$.

We thus conclude  that  the dual map 
$$H_3\big(E(1)\big) \simeq \big(H^3(E(1))\big)^* \longrightarrow \big(H^3(E)\big)^* \simeq H_3(E)$$ 
is also an isomorphism. Since Poincar\'e duality on each component of $E(1)$ induces an isomorphism between
$H^1(E(1))$ and $H_3(E(1))$, we finally obtain that the composition 
$$ H^1(E(1)) \stackrel{P.D.}{\longrightarrow} H_3(E(1)) \longrightarrow H_3(E) \stackrel{T}{\longrightarrow} H^3(E) \longrightarrow H^3(E(1))$$ 
is an isomorphism. The map $\beta\circ\alpha=\overline{\beta}\circ\overline{\alpha}$ is a Hodge summand of this map, hence it is an isomorphism as well.
\end{example}

In the general case, we again consider the bilinear pairing given by
$$H_{n}(E) \stackrel{j_*}{\longrightarrow} H_{n}(\overline{Y}) \stackrel{P.D.}{\longrightarrow} H^{n}(\overline{Y}) \stackrel{j^*}{\longrightarrow} H^{n}(E) \simeq \big(H_{n}(E)\big)^*,$$ 
where $j\colon E \hookrightarrow \overline{Y}$ is the inclusion, and we denote by
$T \colon H_{n}(E) \to H^{n}(E)$ the composition of these maps. 

Specializing \cite[Theorem~2.1.10]{cm05} to our particular situation of an isolated singularity says that this
pairing is nondegenerate as well, that is, $T$ is an isomorphism.
Indeed, since $E$ is compact Borel-Moore homology coincides with singular homology, and the refined intersection form 
$H_{n, 0} (E) \to H^n_0 (E)$ in \cite[Theorem~2.1.10]{cm05}, whose construction is analogous to that of $T$, is an isomorphism; here the index $0$ denotes the $0^{\rm th}$ graded quotient in the perverse filtration on the two sides. Now as described in \cite[Corollary 2.1.12]{cm05}, in the case of a log resolution of an isolated singularity, we have $H^{n}(E)= H^{n}_0(E)$. 
On the other hand, since  $H_{n,0}(E)$ is a subquotient of $H_{n}(E)$,  by dimension reasons we must have  
$H_{n,0}(E) = H_{n}(E)$ as well. Therefore in this case the theorem says precisely that the map $T$ is an isomorphism.

We are  now ready to prove the main result of the section.

\begin{proof}[Proof of Proposition \ref{alpha}]
Consider the composition
$$ H^{n-2}\big(E(1)\big) \stackrel{P.D.}{\longrightarrow} H_{n}\big(E(1)\big) \stackrel{k_*}{\longrightarrow} H_{n}(E) \stackrel{T}{\longrightarrow} H^{n}(E) 
\stackrel{k^*}{\longrightarrow} H^{n}\big(E(1)\big),$$
where $k\colon E(1) \to E$ is the inclusion on each component. In this sequence of maps, only $k_*$ and $k^*$ are potentially not isomorphisms.

Using the fact that $\dim{E(2)} = n-2$, we see that the sequence that computes the $H^{1, n-1}$ part of the weight $n$ cohomology of $E$ is
$$0\to H^{1,n-1}\big(E(1)\big) \to 0.$$
Since $H^{n}(E)$ has a pure Hodge structure of weight $n$, we conclude that 
$$H^{1,n-1}(E) \stackrel{k^*}{\longrightarrow}  H^{1,n-1}\big(E(1)\big)$$ is an isomorphism. 

We can define the dual Hodge structure on $H_{n}\big(E(1)\big)$ by transferring that on $H^{n}\big(E(1)\big)$, and we obtain that 
$$H_{-(n-1),-1}\big(E(1)\big) \stackrel{k_*}{\longrightarrow}   H_{-(n-1),-1}(E)$$ 
is an isomorphism. With respect to these Hodge structures, Poincar\'e duality is an isomorphism of degree $\big(-(n-1),-(n-1)\big)$ on $E(1)$, hence 
$H^{0,n-2}\big(E(1)\big)$ is mapped to 
$H_{-(n-1), -1}\big(E(1)\big)$.  Using that Poincar\'e duality is an isomorphism of degree $(-n,-n)$ on $\overline{Y}$ we conclude that $T$ is a map of degree $(n,n)$.  
Putting everything together, restricting the composition of maps at the beginning of the proof
 to $H^{0,n-2}\big(E(1)\big)$ gives an isomorphism with  $H^{1,n-1}\big(E(1)\big)$. But this restriction is precisely $\overline{\beta}\circ\overline{\alpha}=\beta\circ\alpha$. 
\end{proof}

This completes the proof of Theorem~\ref{main}.

\section{The proof for toric varieties}

Our goal in this section is to show that Conjecture~\ref{conj_vanishing} holds when $Z$ is a toric variety. We note that in this case
it is well known that $Z$ has rational singularities. For the basic facts about toric varieties that we use here, we refer to \cite{Fulton}.

\begin{proof}[Proof of Theorem~\ref{thm_toric}]
It follows from Lemma~\ref{lem1_conj_vanishing} that the assertion in the conjecture is independent of the resolution.
We thus choose a toric resolution of singularities $f\colon Y\to Z$, with reduced exceptional divisor $E$; note that $E$ has simple normal crossings
by default, since it is a 
torus-invariant divisor on a smooth toric variety. 
Let $D=\sum_{i=1}^sD_i$ be the sum of the non-exceptional prime torus-invariant divisors on $Y$. We consider the residue short exact sequence
\begin{equation}\label{short_exact1}
0\longrightarrow \Omega_Y(\log E)\longrightarrow \Omega_Y\big(\log (E+D)\big)\longrightarrow \bigoplus_{i=1}^s\shO_{D_i}\longrightarrow 0.
\end{equation}
Since $\Omega_Y\big(\log (E+D)\big)\simeq\shO_Y^{\oplus n}$, with $n = \dim Z$,  and $Z$ has rational singularities, it follows that
$$R^if_*\Omega_Y\big(\log (E+D)\big)=0\quad\text{for all}\quad i>0.$$
On the other hand, each $f(D_i)$ is a prime torus-invariant divisor on $Z$, hence it is a toric variety, and
$D_i\to f(D_i)$ is a resolution of singularities. 

Suppose first that $n \ge 3$. 
Since $f(D_i)$ has rational singularities, passing to higher direct images in (\ref{short_exact1}) we obtain
$$0=\bigoplus_{i=1}^s R^{n-2}f_*\shO_{D_i}\to R^{n-1}f_*\Omega_Y(\log E) \to R^{n-1}f_*\Omega_Y\big(\log (E+D)\big)=0,$$
and we conclude that $R^{n-1}f_*\Omega_Y(\log E)=0$.

Suppose now that $n=2$. In this case $Z$ has isolated singularities, hence we could apply Theorem~\ref{main}; we prefer to include a direct toric argument.
We may assume that $Z$ is affine, in which case $s=2$.
Let $v_1$ and $v_2$ be the primitive ray generators of the cone defining $Z$, corresponding to $D_1$ and $D_2$ respectively.  
Note that in this case the maps $D_i\to f(D_i)$ are isomorphisms.
If $M$ is the character lattice, then 
$$\Omega_Y\big(\log (E+D)\big)\simeq M\otimes_{\ZZ}\shO_Y,$$
and the long exact sequence in cohomology associated to (\ref{short_exact1}) gives 
$$H^0\big(Y,\Omega_Y\big(\log (E+D)\big)\big)=M\otimes_{\ZZ} H^0(Z,\shO_Z)\overset{\delta}\longrightarrow H^0(D_1,\shO_{D_1})\oplus H^0(D_2,\shO_{D_2})$$
$$\longrightarrow H^1\big(Y,\Omega_Y(\log E)\big)\longrightarrow H^1\big(Y, \Omega_Y(\log (D+E))\big)=0.$$
An easy computation shows that the map $\delta$ is given by
$$u\otimes g\mapsto (\langle u,v_1\rangle (g\circ f)\vert_{D_1},\langle u,v_2\rangle (g\circ f)\vert_{D_2}),$$
hence it is clearly surjective. This implies that $H^1\big(Y,\Omega_Y(\log E)\big)=0$, completing the proof of the theorem.
\end{proof}

\section{Application to the Hodge filtration}

\subsection{Generation level of the Hodge filtration}\label{level}
We now turn to the connection with Saito's filtration on $\shO_X(*D)$. Suppose that $X$ is a smooth complex variety of dimension $n$ and $D$ is a reduced effective divisor on $X$.
We recall that $\shO_X(*D)$ is obtained by localizing $\shO_X$ along $D$. This has a natural module structure over the sheaf of differential operators $\Dmod_X$, and as discussed in the introduction, Saito's theory
of mixed Hodge modules \cite{Saito-MHM} endows it with a Hodge filtration $F_k \shO_X (*D)$, $k \ge 0$, compatible with the order filtration on $\Dmod_X$. Recall that $F_\bullet \shO_X(*D)$ is 
generated at level $k$ if 
$$F_{\ell} \Dmod_X \cdot F_k\shO_X(*D)=F_{k+\ell}\shO_X(*D)\quad\text{for all}\quad \ell\geq 0.$$

Suppose now that $f\colon Y\to X$ is a log resolution of $(X,D)$ that is an isomorphism over $X\smallsetminus D$. If $E=(f^*D)_{\rm red}$, then
it was shown in 
\cite[Theorem~17.1]{MP} that $F_\bullet \shO_X(*D)$ is generated at level $k$ if and only if 
\begin{equation}\label{eqn1_thm}
R^qf_*\Omega^{n-q}_Y(\log E)=0 \quad\text{for all}\quad q>k.
\end{equation}
Based on this criterion, it was shown in \cite[Theorem~B]{MP} that it is always generated at level $n-2$. We will also use it here in order to relate Conjectures~\ref{conj_filtration} and \ref{conj_vanishing}.
Note that the higher-direct images that appear in (\ref{eqn1_thm}) are independent on the resolution $f$; see \cite[Corollary~31.2]{MP}.

\subsection{Proof of Theorem ~\ref{thm_filtration}}
The additional key ingredient in the proof of Theorem~\ref{thm_filtration} is a vanishing result for higher direct images in the case of normal divisors.
We assume that $n\geq 3$ and $D$ is normal. In particular, we have $\dim(D_{\rm sing})\leq n-3$. We consider a log resolution $f\colon Y\to X$
of $(X,D)$ that is a composition of blow-ups with centers contained in the inverse image of $D_{\rm sing}$, and which have simple normal crossings with
the total transform of $D$ on the corresponding model.  If $E=(f^*D)_{\rm red}$, then we write $E=\widetilde{D}+F$, where
$\widetilde{D}$ is the strict transform of $D$ and $F$ is the reduced exceptional divisor.

\begin{proposition}\label{vanishing1}
With the above notation, we have 
$$f_*\Omega^2_Y(\log F)=\Omega_X^2\quad\text{and}\quad R^qf_*\Omega_Y^2(\log F)=0\quad\text{for all}\quad q\geq 1.$$
\end{proposition}

\begin{proof}
By assumption, $f$ can be written as a composition
$$Y=X_r\overset{f_r}\longrightarrow X_{r-1}\overset{f_{r-1}}\longrightarrow\ldots\overset{f_1}\longrightarrow X_0=X,$$
where $f_i$ is the blow-up of $X_{i-1}$ along $W_{i-1}$, with exceptional divisor $G_i$. We denote by $F_i$ the
exceptional divisor of $f_1\circ\ldots\circ f_i$, hence
$$F_i=f_i^*F_{i-1}+G_i.$$
Using the Leray spectral sequence, it is enough to show that for every $i$, with $1\leq i\leq r$, we have
\begin{equation}\label{eq1_vanishing1}
{f_i}_*\Omega^2_{X_i}(\log F_i)=\Omega_{X_{i-1}}^2(\log F_{i-1})\quad\text{and}
\end{equation}
$$ R^q{f_i}_*\Omega_{X_i}^2(\log F_i)=0\quad\text{for all}\quad q\geq 1.$$
If $W_{i-1}\subseteq F_{i-1}$, this follows from \cite[Lemmas 1.2 and 1.5]{EV}; cf also \cite[Theorem~31.1(i)]{MP}. 
Suppose now that $W_{i-1}\not\subseteq F_{i-1}$. In this case $W_{i-1}$
is the strict transform of its image in $X$, hence our assumption on $f$ implies that $\codim(W_{i-1},X_{i-1})\geq 3$. 
Moreover, $W_{i-1}$ has simple normal crossings with $F_{i-1}$; 
since the assertion in (\ref{eq1_vanishing1}) is local on $X_{i-1}$, we may assume that we have an algebraic system of coordinates $x_1,\ldots,x_n$ on 
$X_{i-1}$ such that $W_{i-1}$ is defined by $(x_1,\ldots,x_s)$ and each component of $F_{i-1}$ is defined by some $x_j$, with $j>s$.
Let $T$ be the divisor on $X_{i-1}$ defined by $x_1$. Consider the short exact sequence
$$0\to\Omega^2_{X_{i}}(\log F_{i})\to\Omega_{X_{i}}^2\big(\log (F_{i}+\widetilde{T})\big)\to\Omega^1_{\widetilde{T}}(\log F_{i}\vert_{\widetilde{T}})\to 0,$$
where $\widetilde{T}$ is the strict transform of $T$ on $X_i$. 
It follows from the same references as above that
$$ {f_i}_*\Omega^2_{X_i}\big(\log (F_{i}+\widetilde{T})\big)=\Omega^2_{X_{i-1}}\big(\log (F_{i-1}+T)\big)\quad\text{and}$$
$$ R^q{f_i}_*\Omega^2_{X_i}\big(\log (F_{i}+\widetilde{T})\big)=0\quad \text{for all}\quad
q\geq 1.$$
On the other hand, since $\codim(W_{i-1},X_{i-1})\geq 3$, we have that $F_{i}\vert_{\widetilde{T}}$ is the sum of the exceptional divisor of $h\colon \widetilde{T}\to T$
with the strict transform, with respect to this map, of $F_{i-1}\vert_T$. Therefore it follows from \cite[Theorem~31.1(ii)]{MP} that we have
$$h_*\Omega^1_{\widetilde{T}}(\log F_i\vert_{\widetilde{T}})=\Omega^1_T(\log F_{i-1}\vert_T)\quad\text{and}\quad
R^qh_*\Omega^1_{\widetilde{T}}(\log F_i\vert_{\widetilde{T}})=0\,\,\text{for all}\,\,q\geq 1.$$
The long exact sequence in cohomology for the above short exact sequence gives 
$$R^q{f_i}_*\Omega^2_{X_{i}}(\log F_{i})= 0\quad\text{for all}\quad q\geq 2$$
and an exact sequence
$$0\to {f_i}_*\Omega^2_{X_{i}}(\log F_{i})\to \Omega^2_{X_{i-1}}\big(\log (F_{i-1}+T)\big)\to \Omega^1_T(\log F_{i-1}\vert_T)$$
$$\to R^1{f_i}_*\Omega^2_{X_{i}}(\log F_{i})\to 0.$$
These facts imply the assertions in (\ref{eq1_vanishing1}). 
\end{proof}

With the same notation and assumptions as in Proposition~\ref{vanishing1}, consider the morphism $g\colon \widetilde{D}\to D$ induced by $f$. 
Note that since $D$ is normal, its connected components are irreducible. By hypothesis, the $f$-exceptional divisor $F$ lies over $D_{\rm sing}$,
hence $g$ is a birational morphism, with exceptional divisor $G:=F\vert_{\widetilde{D}}$ (which has simple normal crossings).

\begin{corollary}\label{cor_vanishing1}
With the above notation, the Hodge filtration on $\shO_X(*D)$ 
is generated at level $n-3$ if and only if
$R^{n-2}g_*\Omega_{\widetilde{D}}^1(\log G)=0$.
\end{corollary}

\begin{proof}
It follows from the discussion in \S\ref{level} that the Hodge filtration on $\shO_X(*D)$ is generated at level $n-3$ if and only if 
$$R^{n-2}f_*\Omega^2_Y(\log E)=0.$$
Consider the exact sequence
$$
0\longrightarrow\Omega_Y^2 (\log F)\longrightarrow\Omega_Y^2(\log E)\longrightarrow\Omega_{\widetilde{D}}^1(\log G)\longrightarrow 0.
$$
As a consequence of Proposition~\ref{vanishing1} we have
$$R^qf_*\Omega_Y^2(\log E)\simeq R^qg_*\Omega_{\widetilde{D}}^1(\log G)\quad\text{for every}\quad q\geq 1,$$
which implies the assertion.
\end{proof}

\begin{proof}[Proof of Theorem~\ref{thm_filtration}]
Since $D$ has rational singularities, it is normal.
We construct a log resolution $f\colon Y\to X$ of $(X,D)$ as in Proposition~\ref{vanishing1}.
Let $F$ be the exceptional divisor of $f$, and $\widetilde{D}$ the strict transform of $D$. We have
seen that the restriction $g\colon\widetilde{D}\to D$ is a resolution of $D$, with exceptional divisor $G=F\vert_{\widetilde{D}}$.
By Corollary~\ref{cor_vanishing1}, 
the Hodge filtration on $\shO_X(*D)$ 
is generated at level $n-3$ if and only if
$R^{n-2}g_*\Omega_{\widetilde{D}}^1(\log G)=0$,
which is equivalent to saying that Conjecture~\ref{conj_vanishing} holds for all connected components of $D$
(recall that by Lemma~\ref{lem1_conj_vanishing}, the assertion in Conjecture~\ref{conj_vanishing} is independent of the chosen resolution).
This shows that Conjecture~\ref{conj_vanishing} holds in the hypersurface case if and only if Conjecture~\ref{conj_filtration} does.
In particular, it follows from Theorem~\ref{main} that Conjecture~\ref{conj_filtration} holds when the divisor $D$ has isolated singularities.
\end{proof}

\section{Conjectural reduction to the case of isolated singularities}

\subsection{A conjecture on Hodge ideals and ${\mathfrak m}$-adic approximation}
If $D$ is a reduced effective divisor on the smooth complex variety $X$, then Saito's Hodge filtration on $\shO_X(*D)$ has the form
$$
F_k\shO_X(*D)=I_k(D)\otimes_{\shO_X}\shO_X\big((k+1)D\big)\quad\text{for all}\quad k\geq 0,
$$
where $I_k(D)$ is a coherent ideal in $\shO_X$, the $k^{\rm th}$ Hodge ideal of $D$. It is known, for example,
that 
$$I_0(D)=\I \big(X, (1-\epsilon)D \big)\quad\text{for}\quad 0<\epsilon\ll 1,$$
where $\I (X,\alpha D)$ is the multiplier ideal of the $\RR$-divisor $\alpha D$.
For these and other basic facts about Hodge ideals, we refer to \cite{MP}; for the definition of multiplier ideals,
see \cite[Chapter~9]{Lazarsfeld}.

We propose the following conjecture regarding the behavior of Hodge ideals with respect to ${\mathfrak m}$-adic approximation. 

\begin{conjecture}\label{conj_m_adic}
Let $D$ be a reduced effective divisor on the smooth complex variety $X$, and let $k$ be a non-negative integer.
 If $x\in X$ is a point defined by the ideal ${\mathfrak m}_x$,
then for every $r\geq 1$ there exists a positive integer $q(r)$ such that for every reduced effective divisor $E$ on $X$, with
$$\shO_X(-E)\subseteq\shO_X(-D)+{\mathfrak m}_x^{q(r)},$$ we have
$$I_k(E)\subseteq I_k(D)+{\mathfrak m}_x^r.$$
\end{conjecture}

\begin{example}\label{multiplier}
The assertion in the conjecture holds for $k=0$. Indeed, let $\epsilon>0$ be such that $I_0(D)=\I\big(X, (1-\epsilon)D \big)$.
We claim that if $\dim X=n$, then we may take $q(r)$ to be any integer such that
$q(r)>\frac{n+r-1}{\epsilon}$. In order to see this, 
choose $\eta$ small enough, with $0<\eta<\epsilon$, such that $\epsilon-\eta>\frac{n+r-1}{q(r)}$.
It is enough to show that for every such $\eta$ and every reduced effective divisor $E$ with
$\shO_X(-E)\subseteq\shO_X(-D)+{\mathfrak m}_x^{q(r)}$, we have 
$$
\I \big(X, (1-\eta)E\big)\subseteq I_0(D)+{\mathfrak m}_x^r.
$$
By using the Summation theorem (see \cite{Takagi} or \cite{JM}), for every such $E$ we have
$$\I \big(X, (1-\eta)E\big)\subseteq \I\big(X,(\shO_X(-D)+{\mathfrak m}^{q(r)}_x)^{1-\eta}\big) = $$ 
$$=\sum_{\gamma+\delta=1-\eta}\I \big(X, \shO_X(-D)^{\gamma}\cdot{\mathfrak m}_x^{\delta q(r)}\big) \subseteq$$
$$\subseteq \I\big(X, (1-\epsilon)D\big)+\I({\mathfrak m}_x^{(\epsilon-\eta)q(r)})=I_0(D)+{\mathfrak m}_x^{\lfloor (\epsilon-\eta)q(r)\rfloor-n+1}
\subseteq I_0(D)+{\mathfrak m}_x^r,$$
where we used the fact that 
$$\I(X,{\mathfrak m}_x^{\alpha})={\mathfrak m}_x^{\lfloor \alpha\rfloor-n+1}\quad\text{for every}\quad \alpha\geq 0,$$
with the convention that ${\mathfrak m}_x^\ell=\shO_X$ for $\ell\leq 0$ (see  \cite[Example~9.2.14]{Lazarsfeld}).
\end{example}

\subsection{Reduction to isolated singularities}
The interest in Conjecture~\ref{conj_m_adic} comes from the fact that a positive answer would allow one to reduce the proof 
of general properties of Hodge ideals to the case when $D$ has only isolated singularities. We illustrate this by showing that it allows a reduction of Conjecture~\ref{conj_filtration} to this case, which is treated in Theorem~\ref{thm_filtration}.

\begin{theorem}\label{thm_m_adic}
If Conjecture~\ref{conj_m_adic} holds, then Conjecture~\ref{conj_filtration} holds as well.
\end{theorem}

\begin{proof}
In order to check Conjecture~\ref{conj_filtration}, we may assume that $X$ is an affine variety
of dimension $n\geq 3$ and $D$ is defined by $f\in\shO_X(X)$.
Since we already know that the filtration on $\shO_X(*D)$
is generated at level $n-2$ by \cite[Theorem~B]{MP}, it is generated at level $n-3$ if and only if
\begin{equation}\label{eq_thm_m_adic}
F_{n-2}\shO_X(*D)\subseteq F_1\Dmod_X\cdot F_{n-3}\shO_X(*D).
\end{equation}
(The opposite inclusion always holds, since the filtration on $\shO_X(*D)$ is compatible with the order filtration on 
$\Dmod_X$.) It is enough to show that the inclusion (\ref{eq_thm_m_adic})
holds at every point $p\in D$, as the assertion is trivial away from $D$. We fix $p\in D$ and, after possibly replacing $X$
by a smaller neighborhood of $p$, we assume that there is an algebraic system of coordinates $x_1,\ldots,x_n$
on $X$ that generate the ideal ${\mathfrak m}_p$ defining $p$. A straightforward computation shows that the right-hand side of
(\ref{eq_thm_m_adic}) is equal to $J_{n-2}(D)\otimes\shO_X\big((n-1)D\big)$, where $J_{n-2}(D)$ is the ideal generated by
$$\left\{f\frac{\partial h}{\partial x_i}-(n-2)h\frac{\partial f}{\partial x_i}\mid h\in \Gamma\big(X,I_{n-3}(D)\big), 
1\leq i\leq n\right\}.$$
We thus have $J_{n-2}(D)\subseteq I_{n-2}(D)$, and we need to show that the opposite inclusion holds at $p$. By Krull's Intersection Theorem, it suffices  to show that
\begin{equation}\label{eq2_thm_m_adic}
I_{n-2}(D)\subseteq J_{n-2}(D)+{\mathfrak m}_p^r\quad\text{for all}\quad r\geq 1.
\end{equation}

Given $r\geq 1$, we apply the assertion in Conjecture~\ref{conj_m_adic} to choose $q\geq r$ such that for every $g\in (f)+{\mathfrak m}_p^q$,
if $E$ is the divisor generated by $g$, then 
\begin{equation}\label{eq3_thm_m_adic}
I_{n-3}(E)\subseteq I_{n-3}(D)+{\mathfrak m}_p^{r+1}.
\end{equation}

We choose 
$$g_{\lambda}=\lambda_0f+\sum_{i=1}^n\lambda_ix_i^q,$$
where $\lambda=(\lambda_0,\ldots,\lambda_n)\in\CC^{n+1}$ is general. 
Let $E_{\lambda}$ be the divisor defined by $g_{\lambda}$.
Note that $E_{\lambda}$ has an isolated singularity at $p$
(in particular, it is reduced).
Indeed, the base locus of the linear system generated by $f,x_1^q,\ldots,x_n^q$ is equal to $\{p\}$; we deduce from the Kleiman-Bertini theorem that for $\lambda$ general, $E_{\lambda}$ is smooth away from $p$. Moreover, $E_{\lambda}$ has a rational singularity at $p$; indeed, this is the case for
$\lambda=(1,0,\ldots,0)$ by assumption, hence the assertion for general $\lambda$ follows from Elkik's result on deformations of rational singularities (see \cite[Th\'{e}or\`{e}me~4]{Elkik}). We can therefore apply Theorem~\ref{thm_filtration} to 
$E_\lambda$, in order to conclude that
$$I_{n-2}(E_{\lambda})=J_{n-2}(E_{\lambda}).$$

On the other hand,
since $g_\lambda \in (f)+{\mathfrak m}_p^q$, we deduce from (\ref{eq3_thm_m_adic}) and
the definition of the ideals $J_{n-2}$
that 
$$J_{n-2}(E_{\lambda})\subseteq J_{n-2}(D)+{\mathfrak m}_p^r.$$

We thus conclude that in order to complete the proof of (\ref{eq2_thm_m_adic}), it is enough to show that
if $U\subseteq\CC^{n+1}$ is any open subset such that $E_{\lambda}$ is reduced for every $\lambda\in U$, then
\begin{equation}\label{eq10}
I_{n-2}(D)\subseteq \sum_{\lambda\in U}I_{n-2}(E_{\lambda}).
\end{equation}
To see this, consider $Y=X\times \CC^{n+1}$, and $h=y_0f+\sum_{i=1}^ny_ix_i^q$, defining a divisor $H$,
where $y_0,\ldots,y_n$ are the coordinates on $\CC^{n+1}$. It follows from \cite[Theorem~16.1, Remark~16.8]{MP} that after possibly
replacing $U$ by a smaller open subset, we may assume that 
\begin{equation}\label{eq11}
I_{n-2}(E_{\lambda})=I_{n-2}(H)\vert_{y=\lambda},
\end{equation}
where the right-hand side denotes the image of $I_{n-2}(H)$ via the morphism  $\shO_X (X)[y_0,\ldots,y_n]\to\shO_X (X)$ 
of $\shO_X(X)$-algebras that maps
$y_i$ to $\lambda_i$ for all $i$. On the other hand, the Restriction Theorem for Hodge ideals (see \cite[Theorem~A]{MP2}) says that
the inclusion 
$$
I_{n-2}(E_{\lambda})\subseteq I_{n-2}(H)\vert_{y=\lambda}
$$
holds for all $\lambda$. In particular, taking
$\lambda=(1,0,\ldots,0)$ we see that
\begin{equation}\label{eq12}
I_{n-2}(D)\subseteq I_{n-2}(H)\vert_{y=(1,0,\ldots,0)}.
\end{equation}
It is an elementary exercise to see that for every $P\in\shO_X(X)[y_0,\ldots,y_n]$ and for every $(a_0,\ldots,a_n)\in\CC^{n+1}$,
$P(a_0,\ldots,a_n)\in\shO_X(X)$ lies in the linear span of $\{P(\lambda)\mid\lambda\in U\}$. This observation, in combination with 
(\ref{eq11}) and (\ref{eq12}), gives the inclusion (\ref{eq10}), completing the proof of the theorem.
\end{proof}

\end{document}